\documentclass[11pt,reqno]{amsart}

\usepackage{amsmath, amsthm, amssymb, esint}
\usepackage{amsfonts}

\usepackage{hyperref}

\usepackage[ansinew]{inputenc}
\usepackage[dvips]{epsfig}
\usepackage{graphicx}
\usepackage[english]{babel}

\usepackage{thmtools}
\theoremstyle{plain}
\declaretheorem[title=Theorem, parent=section]{theorem}
\declaretheorem[title=Lemma,sibling=theorem]{lemma}
\declaretheorem[title=Proposition,sibling=theorem]{proposition}

\theoremstyle{definition}

\declaretheorem[title=Remark,sibling=theorem]{remark}
\declaretheorem[title=Remark, numbered=no]{remark*}

\declaretheorem[title=Assumption, numbered=no]{assumption*}

\numberwithin{equation}{section}

\usepackage[backgroundcolor=white, bordercolor=blue,
linecolor=blue]{todonotes}

\usepackage{dsfont}
\usepackage{bbm}

\newcommand{\R}{\mathbb{R}}

\newcommand{\cM}{\mathcal{M}}

\newcommand{\eps}{\varepsilon}

\newcommand{\average}{{\mathchoice {\kern1ex\vcenter{\hrule height.4pt
width 6pt depth0pt} \kern-9.7pt} {\kern1ex\vcenter{\hrule
height.4pt width 4.3pt depth0pt} \kern-7pt} {} {} }}

\begin{document}
\allowdisplaybreaks
 \title[$L^p$ estimates for the Laplacian via blow-up]{$L^p$ estimates for the Laplacian via blow-up}

\author{Jan Lewenstein-Sanpera}
\author{Xavier Ros-Oton}

\address{Universitat de Barcelona, Gran Via de les Corts Catalanes 585, 08007 Barcelona, Spain}
\email{jlewensa7@alumnes.ub.edu}

\address{ICREA, Pg. Llu\'is Companys 23, 08010 Barcelona, Spain \& Universitat de Barcelona, Departament de Matem\`atiques i Inform\`atica, Gran Via de les Corts Catalanes 585, 08007 Barcelona, Spain \& Centre de Recerca Matem\`atica, Barcelona, Spain}
\email{xros@icrea.cat}



\allowdisplaybreaks

\begin{abstract}
In this note we provide a new proof of the $W^{2,p}$ \textit{Calder\'on-Zygmund} regularity estimates for the Laplacian, i.e., $\Delta u=f$ and its parabolic counterpart $\partial_t u-\Delta u=f$. Our proof is an adaptation of a contradiction and compactness argument that so far had been only used to prove estimates in H\"older spaces.
This new approach is simpler than previous ones, and avoids the use of any interpolation theorem.
\end{abstract}
 
\allowdisplaybreaks

\maketitle

\section{Introduction}

The main concern for this article is the following classical theorem: 
\begin{theorem}[\cite{CZ52}]
 \label{th: calderon-zygmund theorem}
    Let $1<p<\infty$, and $u\in H^1(B_{1})$ be any weak solution to:
    \begin{equation*}
        \Delta u = f \quad in \quad B_{1},
    \end{equation*}
    with $f\in L^{p}(B_{1})$. Then u is in $W^{2,p}$ inside $B_{1}$ and the following estimate holds:
    \begin{equation}
        \int_{B_{1/2}}|D^{2}u|^{p} \leq C\left(\int_{B_1} |u|^{p} +\int_{B_1}|f|^{p}\right).
        \label{eq: lp_elliptic_estimate}
    \end{equation}
The constant $C$ depends only on $n$ and $p$.
\end{theorem}

The classical proof of this result, which can be found in \cite{GT}, is as follows:
\begin{itemize}
\item[(i)] Prove the result for $p=2$ by an easy integration by parts argument.
\item[(ii)] Using the Calder\'on-Zygmund decomposition, and the maximal function of Hardy-Littlewood, prove a \emph{weak} $L^1$ estimate corresponding to the case $p=1$.
\item[(iii)] Prove the Marcinkiewicz interpolation theorem, and use it to deduce the result for all $1<p<2$.
\item[(iv)] By a duality argument, deduce the result for all $2<p<\infty$.
\end{itemize}

This result and its proof are a landmark in elliptic PDE, and connects this subject to harmonic analysis.

Up to date, some other approaches have been introduced to find new proofs of this result.
In particular, an alternative proof is presented in \cite{Le13}, in which (ii)-(iii)-(iv) are replaced by:
\begin{itemize}
\item[(ii')] Prove an $L^\infty$-to-BMO estimate, corresponding to the case $p=\infty$.
\item[(iii')] Prove the Stampacchia interpolation theorem, by using the one due to Marcinkiewicz as well as the sharp maximal function of Fefferman-Stein, and use it to deduce the result for all $2<p<\infty$.
\item[(iv')] By a duality argument, deduce the result for all $1<p<2$.
\end{itemize}

On the other hand, a completely different (and more geometric) approach was developed in \cite{Ca89} to treat fully nonlinear equations for $p>n$, and later in \cite{Wa13} for the Laplacian for all $1<p<\infty$.
The proof in \cite{Wa13} uses only the maximal function, energy estimates, and Vitali covering lemma, to establish a decay estimate for the superlevel sets of $\mathcal M|D^2u|^2$ in terms of those of $\mathcal M|f|^2$ (where $\mathcal M$ is the maximal function).

Our goal in this paper is to give a new proof of this result, which we believe is simpler and easier to follow than the one in \cite{Wa13,Ca89}.
Our proof uses only the sharp maximal function, combined with a quite elementary contradiction and compactness argument in the spirit of those in \cite{Si97,Se15}; see also \cite{FR22}.

In the last decade there have been many regularity estimates that have been established by a contradiction and compactness estimate in the spirit of those in \cite{Si97,Se15}.
However, this type of argument had been always used to prove estimates in H\"older spaces, and the present paper is the first one to establish a $W^{2,p}$ estimate via a blow-up argument.
Notice that $L^p$ regularity is quite different than $C^{0,\alpha}$, which is pointwise. 
To solve this issue we use the sharp maximal function, and establish the pointwise inequality \eqref{eq: elliptic_contradiction_estimate} below.
The idea to establish such a pointwise inequality comes from \cite{LZ24}, where a simple proof of the $L^\infty$ to BMO estimate is given, with a different argument.

Our proof is also somewhat related to the proof of \cite[Theorem 13.20]{Le23}, which uses the sharp maximal function and Caccioppoli's inequality (but not a blow-up and compactness argument).
We refer also to the work \cite{Kr07} and the survey \cite{Do20} for another approach in which the sharp maximal function is used to prove $L^p$ estimates for equations with VMO coefficients, in particular obtaining pointwise estimates for the sharp function of $D^2 u$.

\subsection{Acknowledgements}
X. R. was supported by the European Union under the ERC Consolidator Grant No. 101123223 (SSNSD), by the AEI project PID2021-125021NA-I00 (Spain), by the AEI-DFG project PCI2024-155066-2, by the grant RED2022-134784-T funded by AEI/10.13039/501100011033, by AGAUR Grant 2021 SGR 00087 (Catalunya), and by the Spanish State Research Agency through the Mar\'ia de Maeztu Program for Centers and Units of Excellence in R{\&}D (CEX2020-001084-M).

The authors would like to thank the referee for the careful reading of the manuscript and their comments, which helped to improve the paper.


\section{Elliptic equations}
\label{sec:elliptic-estimate}
In this section we focus on the regularity of the weak solutions to Poisson equation 
    \begin{equation}
        \Delta u = f \quad \textrm{in}\quad  B_{1}.
        \label{eq: poisson_equation}
    \end{equation}
    
\subsection{Sharp maximal function}
We provide a very brief introduction to the tools needed in order to prove Theorem \ref{th: calderon-zygmund theorem}. Further information can be found in \cite{LL01, Ev10, GM12}.

We define the 2-sharp maximal function (see, e.g. \cite{ST89}) for any $w\in L^2(D)$ and any $B_{r_\circ}(x)\subset D$ as
\begin{equation*}
        \mathcal{M}_{2}^{\#}w(x) = \sup_{r\in (0,r_\circ)} \fint_{B_r(x)}\left|w-\overline{w}_{B_r(x)}\right|^2,
\end{equation*}
where $\overline{w}_{E}:=\fint_E w$ and $r_\circ>0$ is a fixed constant.

Then, we have the following estimate for any $2<p<\infty$ and $w\in L^p(B_{\rho+r_\circ})$
\begin{equation}
        c\|w\|_{L^p(B_\rho)} \;\leq\; \|\mathcal{M}_{2}^{\#}w\|^{1/2}_{L^{p/2}(B_\rho)}+ \|w\|_{L^1(B_{\rho+r_\circ})}\;\leq\; C\|w\|_{L^p(B_{\rho+r_\circ})},
       \label{estimate: lp_norms_maximal}
\end{equation}
where $c,C$ depend only on $p$, $n$, $r_\circ$, and $\rho$. 
Indeed, the first inequality follows easily from a classical theorem of {Fefferman and Stein} on the sharp maximal function  \cite{GM12,Ma04} \footnote{More precisely, we use that for any $v\in L^1(U)$ and any $B_{r_\circ}(x)\subset U$ we have $v^{\#}(x) \leq \sup_{r\in (0,r_\circ)} \fint_{B_r(x)}\left|v-\overline{v}_{B_r(x)}\right| +C\|v\|_{L^1(U)}$, and apply this to $v=w|_{B_{\rho+r_\circ}}$ together with the results of \cite{GM12,Ma04} and H\"older's inequality.},
while the second inequality follows from the trivial bound $ \mathcal{M}_{2}^{\#}w\leq  4\mathcal{M} (w^2)$ and the strong $L^p$-$L^p$ inequality for the maximal function \cite{GT}.

When applied to a Hessian matrix, we still denote $ \mathcal{M}_{2}^{\#}D^2 w(x) = \sup_{r\in (0,r_\circ)} \fint_{B_r(x)}\big|D^2w-\overline{D^2w}_{B_r(x)}\big|^2$.

\subsection{Regularity estimates}
We want to establish the estimate \eqref{eq: lp_elliptic_estimate} for $p\neq 2$.
Indeed, recall that in the simplest case $p=2$, i.e., 
    \begin{equation}
       \|D^2 u\|_{L^2(B_{1/2})} \leq C\left(\|u\|_{L^2(B_1)} +\|f\|_{L^2(B_1)}\right),
        \label{p=2}
    \end{equation}
it follows easily from the identity
\[\int_{\R^n} |D^2 v|^2 = \int_{\R^n} |\Delta v|^2 \qquad \forall v\in C^\infty_c(\R^n);\]
see e.g. \cite[Remark 2.13]{FR22}.

The estimate \eqref{eq: lp_elliptic_estimate} $2<p<\infty$ will follow from the following:

\begin{proposition}
\label{prop: elliptic_estimate}
    Let $u,f\in C^{\infty}(B_1)$, with $\Delta u = f$ in $B_1$, and $r_\circ=\frac18$. 
Then,
    \begin{equation}
    \mathcal{M}_{2}^{\#}D^2u(x) \leq C\big(\|u\|^2_{L^2(B_1)} + \|f\|^2_{L^2(B_1)} +  \mathcal{M}_{2}^{\#}f(x)\big),
    \label{eq: elliptic_contradiction_estimate}
    \end{equation}
for any $x\in B_{3/4}$,  where $C$ depends only on $n$.
\end{proposition}

To prove this result, we need the following elementary Lemma.

\begin{lemma}
\label{lemma: sequences theta}
    Let $\{u_k\}_k\subset W^{2,2}(B_{r_\circ})$ be a sequence of functions satisfying $\|D^2u_k\|_{L^2(B_{r_\circ})}\leq C_0$ and 
    \begin{equation*}
        \sup_k \sup_{r\in(0,r_\circ)}\fint_{B_r}|D^2u_k(x) - \overline{D^2u_k}_{B_r}|^2 = \infty.
    \end{equation*}
    Then, for any $0<\delta < 1$ there exists $r_m\to 0$ and a subsequence $k_m$ such that:
    \begin{equation*}
    (1-\delta) \fint_{B_\rho}|D^2u_{k}(x) - \overline{D^2u_{k}}_{B_\rho}|^2 \, \leq \fint_{B_{r_m}}|D^2u_{k_m}(x) - \overline{D^2u_{k_m}}_{B_{r_m}}|^2 \qquad \forall k\in\mathbb N, \ \rho\in [r_m,r_\circ].
    \end{equation*}
\end{lemma}

\begin{proof}
For $r<r_\circ$, the quantity: 
\begin{equation*}
    \Theta(r) = \sup_k \sup_{\rho\in [r,r_\circ]} \fint_{B_\rho}|D^2u_k(x) - \overline{D^2u_k}_{B_\rho}|^2
\end{equation*}
satisfies that $\Theta(r) \to \infty$ as $r\to0$. Hence, for any $\varepsilon>0$ there exists $r_\varepsilon \geq \varepsilon$ and  $k_\varepsilon$ such that:
\begin{equation*}
    (1-\delta)\Theta(\varepsilon) \leq \fint_{B_{r_\varepsilon}}|D^2u_{k_\varepsilon}(x) - \overline{D^2u_{k_\varepsilon}}_{B_{r_\varepsilon}}|^2.
\end{equation*}

Let now $\varepsilon\to 0$. Since $\|D^2u_{k_\varepsilon}\|_{L^2(B_{r_\circ})}$ is bounded, then for $\fint_{B_{r_\varepsilon}}|D^2u_{k_\varepsilon}(x) - \overline{D^2u_{k_\varepsilon}}_{B_{r_\varepsilon}}|^2 \to \infty$ we need $r_\varepsilon\to 0$ as $\varepsilon\to0$. 
Moreover, by monotonicity of $\Theta(r)$, we have that $\Theta(r_\varepsilon)\leq\Theta(\varepsilon)$ with
\begin{equation*}
    (1-\delta)\Theta(r_\varepsilon) \leq \fint_{B_{r_\varepsilon}}|D^2u_{k_\varepsilon}(x) - \overline{D^2u_{k_\varepsilon}}_{B_{r_\varepsilon}}|^2.
\end{equation*}
The lemma follows by choosing any sequence $\varepsilon_m\to0$.
\end{proof}

We next give the: 

\begin{proof}[Proof of Proposition \ref{prop: elliptic_estimate}]
It suffices to establish the result for $x=0$.
    Let us show the result by contradiction through a blow-up method. Suppose that the estimate \eqref{eq: elliptic_contradiction_estimate} does not hold. 
Then, there exist  functions $u_k,f_k\in C^{\infty}(B_1)$, such that $\Delta u_k = f_k$ in $B_1$, and
    \begin{equation*}
    \mathcal{M}_{2}^{\#}D^2u_{k}(0) > k\big(\|u_{k}\|^2_{L^2(B_1)} + \|f_{k}\|^2_{L^2(B_1)} +  \mathcal{M}_{2}^{\#}f_{k}(0)\big)
    \end{equation*}
    for all $k\in \mathbb N$.
Moreover, dividing $u_k$ by a constant if necessary, we may assume
\[\|u_{k}\|^2_{L^2(B_1)} + \|f_{k}\|^2_{L^2(B_1)} +  \mathcal{M}_{2}^{\#}f_{k}(0) \;= 1.\]
It then follows from the previous two inequalities that
    \begin{equation*}
        \sup_k \sup_{r\in(0,r_\circ)}\fint_{B_r}|D^2u_k(x) - \overline{D^2u_k}_{B_r}|^2 = \infty.
    \end{equation*}
Moreover, thanks to \eqref{p=2}, we have that $\|D^2 u_k\|_{L^2(B_{1/2})}\leq C$.
    
    Let $r_m\to 0$ and $k_m$ be given by \autoref{lemma: sequences theta} (applied with $\delta=\frac12$) and define new functions $v_m, g_m \in C^{\infty}(B_{1/r_m})$ as
    \begin{align*}
        v_m(x) &= \frac{u_{k_m}(r_mx) - p_m(x)}{r_m^2\Theta^{1/2}(r_m)}, & g_m(x)&=\frac{f_{k_m}(r_mx)-\overline{f_{k_m}}_{B_{r_m}}}{\Theta^{1/2}(r_m)},
    \end{align*}
where 
\begin{equation}\label{Theta-ineq}
  \fint_{B_\rho}\big|D^2u_{k}(x) - \overline{D^2u_{k}}_{B_\rho}\big|^2 \leq 
2\fint_{B_{r_m}}\big|D^2u_{k_m}(x) - \overline{D^2u_{k_m}}_{B_{r_m}}\big|^2
=: \Theta(r_m)  \qquad \forall k\in \mathbb N,\ \rho\geq r_m.
\end{equation}
Notice also that it follows from the previous inequality that $\Theta(r_m)\to \infty$.

    Here, $p_m(x)$ denotes a quadratic polynomial such that $\overline{v_m}_{B_1} = \overline{\nabla v_m}_{B_1} = \overline{D^2v_m}_{B_1} = 0$. 
Notice that $\Delta v_m = g_m$ in $B_{1/r_m}$.
    
    Let us show that $D^2 v_m$ are bounded in the $L^2$ norm in $B_R$ with $1<R<1/r_m$:
    \begin{multline*}
    \|D^2v_m\|_{L^2(B_R)}^2 \;=\; \frac{\int_{B_{R}}|D^2u_{k_m}(r_mx) - \overline{D^2u_{k_m}}_{B_{r_m}}|^2}{\Theta(r_m)} \;=\; \frac{cR^n\fint_{B_{Rr_m}}|D^2u_{k_m}(y) - \overline{D^2u_{k_m}}_{B_{r_m}}|^2}{\Theta(r_m)}  \\
    \leq\; \frac{2cR^{n}\left[\fint_{B_{Rr_m}}|D^2u_{k_m}(y) - \overline{D^2u_{k_m}}_{B_{Rr_m}}|^2 + |\overline{D^2u_{k_m}}_{B_{Rr_m}} - \overline{D^2u_{k_m}}_{B_{r_m}}|^2\right]}{\Theta(r_m)} \leq\; CR^{n}(1+R^{n}),
    \end{multline*}
  where we used \eqref{Theta-ineq} and 
    \begin{multline*}
    \big|\overline{D^2u_{k_m}}_{B_{Rr_m}} - \overline{D^2u_{k_m}}_{B_{r_m}}\big|^2 \;\leq\; \left(\fint_{B_{r_m}} |D^2u_{k_m}(x) - \overline{D^2u_{k_m}}_{B_{Rr_m}}|\right)^{2} \\
    \leq \; \fint_{B_{r_m}} |D^2u_{k_m}(x) - \overline{D^2u_{k_m}}_{B_{Rr_m}}|^2 \\
    \leq\; R^{n}\fint_{B_{Rr_m}} |D^2u_{k_m}(x) - \overline{D^2u_{k_m}}_{B_{Rr_m}}|^2 \;\leq\; R^{n}\Theta(r_m)
    \end{multline*}
  where we used \eqref{Theta-ineq} again. 
Moreover, using the same arguments as above, one gets
    \begin{equation*}
    \|g_m\|^2_{L^2(B_R)} \leq \frac{R^{n}(1+R^{n})\mathcal{M}_{2}^{\#}f_{k_m}(0)}{\Theta(r_m)}.
    \end{equation*}

    Since $\mathcal{M}_{2}^{\#}f_{k_m}(0) \leq 1$ by hypothesis, then $\|g_m\|^2_{L^2(B_R)} \to 0$ as $m\to\infty$. 
Note as well that
    \begin{equation*}
    \|D^2v_m\|^2_{L^2(B_1)} = \frac{\int_{B_{1}}|D^2u_{k_m}(r_mx) - \overline{D^2u_{k_m}}_{B_{r_m}}|^2}{\Theta(r_m)} =\frac12,
    \end{equation*}
by definition.

Using the estimate \eqref{p=2} (the case $p=2$),  we deduce $\frac12 \leq \|D^2v_m\|_{L^2(B_1)} \leq C(\|v_m\|_{L^2(B_2)} + \|g_m\|_{L^2(B_2)})$ and therefore, for $m$ large enough,
\[\|v_m\|_{L^2(B_2)} \geq c_\circ>0.\] 
Finally, for any fixed $R\in (1,1/r_m)$ we have:
    \begin{multline*}
    \fint_{B_R}|D^2v_m(x)-\overline{D^2v_m}_{B_R}|^{2} \;\leq\; \frac{\fint_{B_{R}}|D^2u_{k_m}(r_mx)-\overline{D^2u_{k_m}(r_mx)}_{B_R}|^2}{\Theta(r_m)} \\
    \leq\; \frac{\fint_{B_{Rr_m}}|D^2u_{k_m}(y)-\overline{D^2u_{k_m}}_{B_{Rr_m}}|^2}{\Theta(r_m)}  \leq 1.
    \end{multline*}
    
For any compact set $K\subset \R^n$, we have proved that the sequence $D^2v_m$ is bounded in $L^2(K)$ --for $m$ large enough so that $K\subset B_{1/r_m}$.
Then, since $\overline{v_m}_{B_1} = \overline{\nabla v_m}_{B_1} =0$, it follows from Poincar\'e inequality that $v_m$ is bounded in $W^{2,2}(K)$.
Hence, up to a subsequence, we have that $v_m\to v$ and $\nabla v_m\to \nabla v$ strongly in $L^2(K)$, while $D^2v_m \to D^2 v$ weakly in $L^2(K)$. 
Moreover, since $g_m\to 0$ strongly in $L^2(K)$, then we can pass the equation $\Delta v_m=g_m$ in $K$ (in its weak formulation) to the limit to deduce that $v$ is harmonic in $K$.
Since this can be done for any compact set $K\subset\R^n$, we find
\[\Delta v = 0 \quad \textrm{in}\quad \mathbb{R}^{n}.\]
    
Using that $v_m\rightharpoonup v$ weakly in $W^{2,2}$ (and strongly in $L^2$), we find $\overline{v_m}_{B_1}\to \overline{v}_{B_1}$, $  \overline{\nabla v_m}_{B_1} \to \overline{\nabla v}_{B_1}$, and $ \overline{D^2v_m}_{B_1}\to \overline{D^2v}_{B_1}$, so that
    \begin{equation}
        \overline{v}_{B_1} = \overline{\nabla v}_{B_1} = \overline{D^2v}_{B_1} = 0,
        \label{cond:1}
    \end{equation}
    \begin{equation}
        \|v\|_{L^2(B_2)}\geq c_\circ>0.
        \label{cond:2}
    \end{equation}
Moreover, using the lower semicontinuity of the $W^{2,2}$ norm under weak convergence, we find
    \begin{equation}
        \|D^2v\|^2_{L^2(B_R)} \;\leq\; CR^{n}(1+R^{n}) \text{  for any  } R>1,
        \label{cond:3}
    \end{equation}
    \begin{equation}
        \fint_{B_R}|D^2v(x)-\overline{D^2v}_{B_R}|^{2} \;\leq\; 1 \text{  for any   } R>1.
        \label{cond:4}
    \end{equation}
    
    Using the {mean value property}, for any $x\in B_{R/2}$ and any $R>1$ we have
    \begin{equation*}
    |D^2 v(x)| \;\leq\; \fint_{B_{R/2}(x)}|D^2 v| \;\leq\; \frac{C}{R^{n/2}}\|D^2 v\|_{L^2(B_{R})} \leq CR^{3n/2}.
    \end{equation*}
Hence, $D^2 v$ is a harmonic function with polynomial growth, and the Liouville theorem implies that it {is a polynomial}.
However, by the first condition \eqref{cond:1} and again the mean value property,
    \begin{equation*}
        \fint_{B_R} D^2v(x) = \fint_{B_1} D^2v(x) = 0,
    \end{equation*}
 meaning that the last condition \eqref{cond:4} is $\fint_{B_R} |D^2v|^2 \leq 1$ for any  $R>1$.
This means that the polynomial $D^2v$ must be identically zero, which means that $v$ is an affine function.
Using again \eqref{cond:1} we reach that $v\equiv 0$, which contradicts \eqref{cond:2}.
Therefore, the estimate \eqref{eq: elliptic_contradiction_estimate} must hold.
\end{proof}

Finally, we provide the:

\begin{proof}[Proof of  \autoref{th: calderon-zygmund theorem}]
We split the proof into two steps.

    \textbf{Step 1.} We first prove the result for $2<p<\infty$.

Let $u$ be a solution in $B_{1}$ from $\Delta u = f$, with $f\in L^{p}(B_{1})$. Let $\eta\in C^{\infty}_{c}(B_{1})$ be any smooth function with $\eta\geq0$ and $\int_{B_{1}}\eta = 1$ and $\eta_{\varepsilon} = \varepsilon^{-n}\eta(x/{\varepsilon})$.
    
Then, the function $u_{\varepsilon}(x) = u*\eta_{\varepsilon}(x)$ is $C^{\infty}$ and satisfies
    \begin{equation*}
        \Delta u_{\varepsilon} = f*\eta_{\varepsilon} =: f_{\varepsilon} \quad \textrm{in} \quad B_{1-\varepsilon}. 
    \end{equation*}
    Since $u_{\varepsilon}$ is $C^{\infty}$ we can use the estimates \eqref{estimate: lp_norms_maximal}, \eqref{p=2}, and \autoref{prop: elliptic_estimate} (rescaled) to get:
    \begin{multline*}
        \|D^2u_{\varepsilon}\|_{L^p(B_{1/2})} \;\leq\; C\big(\|\mathcal{M}_{2}^{\#}D^2u_{\varepsilon}\|^{1/2}_{L^{p/2}(B_{1/2})} + \|D^2u_{\varepsilon}\|_{L^1(B_{3/4})}\big) \\
        \leq\;  C\big(\|u_{\varepsilon}\|_{L^{2}(B_{1-\varepsilon})} + \|f_{\varepsilon}\|_{L^{2}(B_{1-\varepsilon})} + \|\mathcal{M}_{2}^{\#} f_{\varepsilon}\|^{1/2}_{L^{p/2}(B_{1/2})}\big).\\
        \leq\;  C\big(\|u_{\varepsilon}\|_{L^{2}(B_{1-\varepsilon})} + \|f_{\varepsilon}\|_{L^{p}(B_{1-\varepsilon})}\big).
    \end{multline*} 
    
    Thanks to {Young's convolution inequality} for $L^{p}$ norms we have $\|u_{\varepsilon}\|_{L^{2}(B_{1-\varepsilon})}\leq \|u\|_{L^{2}(B_{1})}$ and also $\|f_{\varepsilon}\|_{L^{p}(B_{1-\varepsilon})}\leq\|f\|_{L^{p}(B_{1})}$.
    The result then follows by letting $\eps\to0$ and using the lower semicontinuity of the $W^{2,p}$ norm.

\vspace{2mm} 

    \textbf{Step 2.} The result for $1<p<2$ then follows from Step 1 and a standard duality argument, which we sketch next. 
First, we prove the estimate 
\[\|D^2\bar u\|_{L^p(B_1)} \leq C\|f\|_{L^{p}(B_1)},\qquad 1<p<2\]
for any $f\in C^{\infty}_c(B_1)$, where $\bar u\in C^\infty(\R^n)$ is the solution of $\Delta \bar u = f$ given by the fundamental solution of the Laplacian.

Recall that the $L^p$-norm of a function $f$ can be characterized by
    \begin{equation*}
        \|f\|_{L^p(B_1)} = \sup_{g\in C^\infty_c(B_1)}\left\{\int_{\R^n} fg : \|g\|_{L^{p'}(B_1)}=1\right\}.
    \end{equation*}
    Given $\bar u$ as above, consider a test function $g\in C^{\infty}_c(B_1)$ and let $v\in C^\infty(\R^n)$ be the solution of $\Delta v = g$ given by the fundamental solution of the Laplacian. Then, integrating by parts we get
    \begin{equation*}
        \int_{\mathbb{R}^n} D^2\bar u\, g = \int_{\mathbb{R}^n} f \,D^2v = \int_{B_1} f \,D^2v \leq \|f\|_{L^p(B_1)} \|D^2v\|_{L^{p'}(B_1)} \leq C \|f\|_{L^p(B_1)} \|g\|_{L^{p'}(B_1)},
    \end{equation*}
where the last inequality $\|D^2v\|_{L^{p'}(B_1)} \leq C\|g\|_{L^{p'}(B_1)}$ follows from Step 1, since $\|g\|_{L^{p'}(B_2)}=\|g\|_{L^{p'}(B_1)}$ and $p'>2$.

Taking the supremum one reaches $\|D^2\bar u\|_{L^p(B_1)} \leq C \|f\|_{L^p(B_1)}$, as claimed.
Moreover, by the same approximation argument as in Step 1, the same inequality holds for any weak solutions of $\Delta \bar u=f$ in $\R^n$, with $f\in L^p(B_1)$.

Finally, if $\Delta u=f$ in $B_1$, we consider the global solution $\bar u$ of $\Delta \bar u=f\chi_{B_1}$ in $\R^n$ (given by the fundamental soluton), and then $u-\bar u$ is harmonic in $B_1$. Combining interior regularity estimates for harmonic functions in $B_1$, with the fact that $\|D^2\bar u\|_{L^p(B_1)} \leq C \|f\|_{L^p(B_1)}$, the result follows.
\end{proof}

\begin{remark}
Once we have the $W^{2,p}$ estimates for the Laplacian, one can prove by the method of ``freezing coefficients'' that the same result holds for general operators in non-divergence form with continuous coefficients, i.e.,
\begin{equation*}
    tr\big(A(x)D^2u(x)\big) = \sum_{i,j=1}^{n}a_{ij}(x)\partial_{ij}u(x) = f(x) \quad \textrm{in}\quad B_1,
\end{equation*}
where $A(x) = (a_{ij}(x))_{i,j}$ is uniformly elliptic and continuous; see e.g. \cite{GT,FR22}. 
\end{remark}

\section{Parabolic Equations}
\label{sec:parabolic-estimate}

Our method can also be used to give a new proof for the following parabolic $W^{2,p}$ estimate:

\begin{theorem}
\label{th: calderon-zygmund theorem parabolic}
    Let $1<p<\infty$ and $u$ be any weak solution of
    \begin{equation*}
        \partial_t u - \Delta u = f \; in \; Q_{1},
    \end{equation*}
    with $f\in L^{p}(Q_{1})$. Then $D^2u$ and $u_t$ are in $L^p$ inside $Q_{1}$ and the following estimate holds:
    \begin{equation}
        \int_{Q_{1/2}}|D^{2}u|^{p} + |\partial_t u|^{p} \leq C\left(\int_{Q_1} |u|^{p} +\int_{Q_1}|f|^{p}\right),
        \label{eq: lp_parabolic_estimate}
    \end{equation}
where $C$ depends only on $n$.
\end{theorem}

Here,  $Q_r(x_\circ,t_\circ) = B_r(x_\circ) \times (t_\circ-\frac{r^2}{2},t_\circ+\frac{r^2}{2}]$ represents the parabolic cube. 
As before, the case $p=2$ follows from an easy integration by parts argument \cite{Ev10}, and we will give a new proof for $p\neq2$. 

\subsection*{Parabolic sharp maximal function}
We define the parabolic 2-sharp maximal function for any $w\in L^2(D)$ and any $Q_{r_\circ}(x,t)\subset D$ as
\begin{equation*}
    \cM^{\#}_{2,{\rm par}} w(x,t) = \sup_{r\in(0,r_\circ)}\fint_{Q_r(x,t)}|w-\overline{w}_{Q_r(x,t)}|^2,
\end{equation*}
for which the analogous bounds to \eqref{estimate: lp_norms_maximal} in $Q_r$ hold; see \cite{GM12}.
Here, $r_\circ>0$ is a fixed constant.

\subsection*{Regularity estimates}
In order to prove \autoref{th: calderon-zygmund theorem parabolic} we follow the steps as in the elliptic case. First, we will show an equivalent result for the non-homogeneous heat equation: 
\begin{proposition}
    \label{prop: parabolic_estimate}
    Let $u,f\in C^{\infty}(Q_1)$, with $u_t - \Delta u = f$, and $r_\circ=\frac18$. Then,
    \begin{equation}
    \cM^{\#}_{2,{\rm par}} \partial_t u(x,t) + \cM^{\#}_{2,{\rm par}}D^2u(x,t) \leq C\big(\|u\|^2_{L^2(Q_1)} + \|f\|^2_{L^2(Q_1)} +  \cM^{\#}_{2,{\rm par}}f(x,t)\big),
    \label{eq: parabolic_contradiction_estimate}
    \end{equation}
for any $(x,t)\in Q_{1/2}$, where $C$ depends only on $n$.
\end{proposition}

\begin{proof}
    Following the same steps and calculations as in the elliptic case, we will see through contradiction that the result holds. Assume as before that $(x_\circ,t_\circ)=(0,0)$, and suppose that the estimate \eqref{eq: parabolic_contradiction_estimate} does not hold. Then, there exist a set of functions $u_k,f_k\in C^{\infty}(Q_1)$, such that $\partial_tu_k - \Delta u_k = f_k$ in $Q_1$ and:
    \begin{equation*}
    \cM^{\#}_{2,{\rm par}}\partial_tu_k(0,0) + \cM^{\#}_{2,{\rm par}}D^2u_{k}(x_0,t_0) > k(\|u_{k}\|^2_{L^2(B_1)} + \|f_{k}\|^2_{L^2(B_1)} +  \cM^{\#}_{2,{\rm par}}f_{k}(0,0)).
    \end{equation*}
Moreover we can take $\|u_{k}\|^2_{L^2(B_1)} + \|f_{k}\|^2_{L^2(B_1)} + \cM^{\#}_{2,{\rm par}}f(0,0)= 1$. 
Rewrite the inequality as:
    \begin{equation*}
        \sup_k \sup_{r\in(0,r_\circ)}\fint_{Q_r}\left[|D^2u_k(x) - \overline{D^2u_k}_{Q_r}|^2 + |\partial_t u_k(x) - \overline{\partial_t u_k}_{Q_r}|^2\right] = \infty,
    \end{equation*} 
    and exactly as in the elliptic case, there exist sequences $r_m\to0$ and $k_m$ such that the new functions
    \begin{align*}
        v_m(x) &= \frac{u_{k_m}(r_mx, r_m^2t) - p_m(x,t)}{r_m^2\Theta^{1/2}(r_m)}, & g_m(x)&=\frac{f_{k_m}(r_mx, r_m^2t)-\overline{f_{k_m}}_{Q_{r_m}}}{\Theta^{1/2}(r_m)}.
    \end{align*}
satisfy 
\[
  \fint_{Q_\rho}\big|D^2u_{k} - \overline{D^2u_{k}}_{Q_\rho}\big|^2+\big|\partial_t u_{k} - \overline{\partial_t u_{k}}_{Q_\rho}\big|^2 \leq 
2\fint_{Q_{r_m}}\big|D^2u_{k_m} - \overline{D^2u_{k_m}}_{Q_{r_m}}\big|^2+\big|\partial_t u_{k_m} - \overline{\partial_t u_{k_m}}_{Q_{r_m}}\big|^2
=: \Theta(r_m)
\]
for all  $k\in \mathbb N$, $\rho\geq r_m$.

    Here, $p_m(x,t)$ denotes a quadratic polynomial in $x$ and linear in $t$ such that $\overline{v_m}_{Q_1} = \overline{\nabla v_m}_{Q_1} = \overline{D^2v_m}_{Q_1} = \overline{\partial_t v_m}_{Q_1} = 0$.
    Notice once more that $\partial_tv_m - \Delta v_m = g_m$ in $Q_{1/r_m}$. 
    
    Through the same calculations as in the elliptic case, we find the following bounds:
    \begin{equation*}
    \|D^2v_m\|_{L^2(Q_R)}^2 + \|\partial_tv_m\|_{L^2(Q_R)}^2  \;\leq\; CR^{2+n}(1+R^{2+n}),
    \end{equation*}

    while for our $g_m$ we have
    \begin{equation*}
    \|g_m\|^2_{L^2(Q_R)} \leq \frac{R^{2+n}(1+R^{2+n})\cM^{\#}_{2,{\rm par}}f_{k_m}(0,0)}{\Theta(r_m)}\longrightarrow 0,
    \end{equation*}
    for any fixed $R\in(1,1/r_m)$.  Note as well that
    \begin{equation*}
    \|D^2v_m\|^2_{L^2(Q_1)} + \|\partial _t v_m\|^2_{L^2(Q_1)} \geq \frac12,
    \end{equation*}
    which leads to $\frac12 \leq \|D^2v_m\|_{L^2(Q_1)} + \|\partial_t v_m\|_{L^2(Q_1)} \leq C(\|v_m\|_{L^2(Q_2)} + \|g_m\|_{L^2(Q_2)})$ thanks to the case $p=2$. 
For a large enough $m$, it means that $\|v_m\|_{L^2(Q_2)}\geq c_\circ>0$. 
Finally, following the same steps, for any fixed $R\in(1,1/r_m)$ we have:
    \begin{equation*}
    \fint_{Q_R}\left[|D^2v_m(x,t)-\overline{D^2v_m}_{Q_R}|^{2} + |\partial_t v_m(x,t) - \overline{\partial_t v_m}_{Q_R}|^2\right] \leq 1.
    \end{equation*}

For any compact set $K\subset \R^n\times \R$, we have proved that the sequences $D^2v_m$ and $\partial_tv_m$ are bounded in $L^2(K)$ --for $m$ large enough so that $K\subset Q_{1/r_m}$.
Then, since $\overline{v_m}_{Q_1}= \overline{\nabla v_m}_{Q_1} =0$, it follows from the parabolic Poincar\'e inequality (see, e.g. \cite[Theorem 19]{Li03}) that $\|v_m\|_{L^2(K)}$ and $\|\nabla v_m\|_{L^2(K)}$ are bounded.
Hence, up to a subsequence, we have that $v_m\to v$ strongly in $L^2(K)$, while $D^2v_m \to D^2 v$,  $\partial_t v_m \to \partial_t v$ and $\nabla v_m\to \nabla v$ weakly in $L^2(K)$. 
Moreover, since $g_m\to 0$ strongly in $L^2(K)$, then we can pass the equation $\partial_t v_m-\Delta v_m=g_m$ in $K$ (in its weak formulation) to the limit to deduce that $v$ solves the heat equation in $K$.
Since this can be done for any compact set $K\subset\R^n\times\R$, we find
\[\partial_t v-\Delta v = 0 \quad \textrm{in}\quad \mathbb{R}^{n}\times\R.\]

 Taking limits we find that
    \begin{equation}
        \overline{v}_{Q_1} = \overline{\nabla v}_{Q_1} = \overline{D^2v}_{Q_1} = \overline{\partial_t v}_{Q_1} = 0,
        \label{cond:1 parabolic}
    \end{equation}
    \begin{equation}
        0 < \|v\|_{L^2(Q_2)},
        \label{cond:2 parabolic}
    \end{equation}
    \begin{equation}
        \|D^2v\|^2_{L^2(Q_R)} + \|\partial_t v\|^2_{L^2(Q_R)}\;\leq\; CR^{2+n}(1+R^{2+n}) \text{  for any  } R>1,
        \label{cond:3 parabolic}
    \end{equation}
    \begin{equation}
        \fint_{Q_R}\left[|D^2v(x,t)-\overline{D^2v}_{Q_R}|^{2} + |\partial_t v(x,t)-\overline{\partial_t v}_{Q_R}|^{2}\right] \;\leq\; 1 \text{  for any  } R>1.
        \label{cond:4 parabolic}
    \end{equation}

    By interior regularity for the heat equation, we can bound (see, e.g. \cite[Section~2.3.3.c]{Ev10})
    \begin{equation*}
        \max_{Q_{R/2}} \,|D^k_xD^l_t \,D^2v(x,t)| \leq \frac{C_{kl}}{R^{k+2l+n+2}}\|D^2v\|_{L^1(Q_R)},
    \end{equation*}
   which with condition \eqref{cond:3 parabolic} means that $D^2 v$ and $\partial_t v$ have to be polynomials of degree $N$ at most. By \autoref{lemma: polynomials}, this would mean that:
   \begin{equation*}
       \fint_{Q_R}\left[|D^2v(x,t)-\overline{D^2v}_{Q_R}|^{2} + |\partial_t v(x,t)-\overline{\partial_t v}_{Q_R}|^{2}\right] \geq cR^{2N}>0 \quad \text{  for any  } R>1.
   \end{equation*}
   which fulfils condition \eqref{cond:4 parabolic} only if $N=0$. This means that $D^2v$ and $\partial_t v$ are constants and that $v$ is a quadratic polynomial in space and linear in time. By condition \eqref{cond:1 parabolic} we reach that $v\equiv 0$ which is a contradiction, and therefore the estimate \eqref{eq: lp_parabolic_estimate} must hold.
\end{proof}

We used this elementary result.

\begin{lemma}
\label{lemma: polynomials}
Let $p(x,t)$ be a polynomial of parabolic degree $N$. 
Then, there exists $c>0$ such that
\begin{equation*}
    \fint_{Q_R}|p-c_R|^2 \geq cR^{2N}
\end{equation*}
for any large enough $R>1$ and any $c_R\in\R$.
\end{lemma}

\begin{proof}
    Write $p = p_0 + ... + p_N$, where each $p_i$ is a parabolically homogeneous polynomial of degree $i = 0, ..., N$. 
By triangle inequality, 
    \begin{equation*}
        \frac{1}{2}|p_N-c_R|^2 - |p-p_N|^2 \leq |p-c_R|^2 \leq 2|p-p_N|^2 + 2|p_N-c_R|^2.
    \end{equation*}

   Notice that $p-p_N$ is of degree $N-1$ and therefore $\fint_{Q_R}|p-p_N|^2 \leq CR^{2N-2}$. On the other hand, simple computations lead to
   \begin{equation*}
       \fint_{Q_R}\left|p_N(x,t)-c_R\right|^2dxdt = \fint_{Q_1}\left|p_N(Ry,R^2s)-c_R\right|^2dyds = R^{2N}\fint_{Q_1}\left|p_N-\frac{c_R}{R^N}\right|^2 \geq c_*R^{2N}>0
   \end{equation*}
for any constant $c_R$, where 
\[c_*:= \inf_{\kappa\in\R}\fint_{Q_1}\left|p_N-\kappa\right|^2.\]
   
   Hence, for large enough $R>1$
   \begin{equation*}
      \fint_{Q_R}\left|p-c_R\right|^2 
      \geq \frac12\fint_{Q_R}\left|p_N-c_R\right|^2-\fint_{Q_R}\left|p-p_N\right|^2 
      \geq \frac12c_*R^{2N} -  CR^{2N-2}
      \geq cR^{2N},
   \end{equation*}
   and the Lemma follows.
\end{proof}

Finally, we sketch the:

\begin{proof}[Proof of \autoref{th: calderon-zygmund theorem parabolic}]
 The proof is essentially the same as in the elliptic case, the only nontrivial difference is in the duality argument.
Namely, once we have the result for $p>2$, we fix $1<p<2$ and proceed as follows.

Consider $\Omega = \mathbb{R}^n\times(0,1)$ and let $u\in C^\infty_c(\R^n\times [0,1])$ be such that $\partial_t u-\Delta u=f$ in $\Omega$, with $u(x,0)=0$.
For any $g\in C^{\infty}_c(\Omega)$ let $v$ be such that $\partial_t v - \Delta v = g$ with $v(x,0)=0$. 
Consider the same problem but backwards in time, i.e.,  $\tilde{g} = g(x, 1-t)$ as well as $\tilde{v} = v(x, 1-t)$ which satisfies $-\partial_t \tilde{v} - \Delta \tilde{v} = \tilde{g}$ and $\tilde{v}(x,1)=0$. 

Then, integrating by parts we find
    \begin{multline*}
        \int_{\Omega} \partial_t u\, \tilde{g} = -\int_{\Omega} \partial_t u \partial_t \tilde{v} - \int_{\Omega} \partial_tu\Delta \tilde{v} = -\int_{\Omega} (\partial_t u - \Delta u) \partial_t\tilde{v} + \left[\int_{\mathbb{R}^n}u\Delta \tilde{v}\right]_0^1 = -\int_\Omega f \partial_t \tilde v \leq \\
        \leq \|f\|_{L^{p}(\Omega)}\|\partial_t \tilde{v}\|_{L^{p'}(\Omega)} = C \|f\|_{L^p(\Omega)}\|\tilde{g}\|_{L^{p'}(\Omega)}.
    \end{multline*} 
    
The last inequality $\|\partial_t\tilde{v}\|_{L^{p'}(\Omega)} \leq C\|\tilde{g}\|_{L^{p'}(\Omega)}$ results since $p'>2$. Finally, take the supremum on the left hand side to reach the desired estimate for $\partial_t u$ is in $L^p(\Omega)$. By definition of $f$, it means that $\Delta u$ is bounded in $L^p(\Omega)$ as well, and therefore that $D^2u$ is in $L^p(\Omega)$. This  yields the result for $1<p<2$ in case of global solutions in $\R^n\times (0,1)$.

Finally, if $\partial_tu -\Delta u=f$ in $Q_1$, we consider the global solution $\bar u$ of $\partial_t \bar u-\Delta \bar u=f\chi_{Q_1}$ in $\R^n\times (0,1)$, and then $u-\bar u$ solves the heat equation in $Q_1$. Combining interior regularity estimates for caloric functions in $Q_1$, with the fact that $\|D^2\bar u\|_{L^p(\mathbb{R}^n\times (0,1))}+\|\partial_t\bar u\|_{L^p(\mathbb{R}^n\times (0,1))} \leq C \|f\|_{L^p(\mathbb{R}^n\times (0,1))}$, the result follows.
\end{proof}

\end{document}